%% file: coverage_sampling.tex
\documentclass[11pt]{article}
\usepackage[utf8]{inputenc}
\usepackage[T1]{fontenc}
\usepackage[sc]{mathpazo}
\usepackage{microtype}
\usepackage{amsmath,xcolor}
\usepackage{amssymb}
\usepackage{amsthm}
\usepackage{bm,float}
\usepackage{dsfont}
\usepackage{authblk}
\usepackage{fullpage,caption,wrapfig}
\usepackage{comment}
\usepackage{mathtools}
\usepackage[shortlabels]{enumitem}
\usepackage{complexity}
\usepackage{tikz}  
\usepackage[backend=bibtex, style=alphabetic, url=false]{biblatex}
\bibliography{refs.bib} 
\usepackage{bbm}
\usepackage{hyperref}
\hypersetup{
    colorlinks=true,
    linkcolor=violet,
    filecolor=magenta,      
    urlcolor=cyan,
    citecolor=blue,
    pdffitwindow=true,
}\usepackage[capitalize, nameinlink]{cleveref}
\usepackage{nag,tikz}
\usetikzlibrary{calc}
\usetikzlibrary{decorations.pathreplacing}

\usepackage{algorithm,algpseudocode}
\usepackage{multicol}
\usepackage[scaled]{helvet} 
\usepackage{thmtools}

\theoremstyle{plain}
\newtheorem{theorem}{Theorem}[section]
\newtheorem{lemma}[theorem]{Lemma}
\newtheorem{corollary}[theorem]{Corollary}
\newtheorem{proposition}[theorem]{Proposition}
\newtheorem{fact}[theorem]{Fact}

\newtheorem{definition}[theorem]{Definition}

\theoremstyle{remark}

\newcommand{\old}[1]{}

\newcommand{\rank}{\text{rank}}

\renewcommand{\R}{\ensuremath{\mathbb R}}
\newcommand{\Z}{\ensuremath{\mathbb Z}}

\def\bone{{\bf 1}}

\def\b1{{\bf 1}}
\def\1{{\bf 1}}

\def\R{\mathbb{R}}

\DeclareMathOperator{\supp}{supp}

\newcommand{\norm}[1]{\|#1\|}
\def\Z{\mathbb{Z}}

\newcommand{\declareperson}[1]{\expandafter\newcommand\csname#1\endcsname[1]{\textcolor{orange}{#1: ##1}}}
\DeclareMathOperator{\mix}{mix}
\DeclareMathOperator{\TV}{TV}
\DeclareMathOperator{\Hom}{Hom}
\declareperson{Dorna}
\declareperson{Shayan}

\title{Complete Log Concavity of  Coverage-Like Functions}
\author{Dorna Abdolazimi\thanks{
Research supported in part by NFS grants CCF-1813135.}
}
\author{Shayan Oveis Gharan\thanks{\href{mailto:shayan@cs.washington.edu}{shayan@cs.washington.edu}. Research supported by Air Force Office of Scientific Research grant FA9550-20-1-0212, NSF grants  CCF-1552097, CCF-1907845,  ONR YIP grant N00014-17-1-2429, and a Sloan fellowship.}} 
\affil{University of Washington}

\begin{document}
\maketitle
\begin{abstract}
	We introduce an expressive subclass of non-negative almost submodular set functions, called strongly 2-coverage functions which include coverage and (sums of) matroid rank functions, and  prove  that the homogenization of the generating polynomial  of any such  function is completely log-concave, taking a step towards characterizing the coefficients of (homogeneous) completely log-concave polynomials. As a consequence we obtain that the ''level sets'' of any such function form an ultra-log concave sequence. 
\end{abstract}

\input{intro}

\input{technical}

\printbibliography
\input{appendix}

\end{document}

%% file: intro.tex
\section{Introduction}

A polynomial $p\in \R[z_1,\dots,z_n]$  is  {\em log-concave} over $\R^n_{\geq 0}$, if $p$ is non-negative and  $\log p$ is a concave function on $\R^n_{\geq 0}$. Note that the identically zero polynomials are log-concave. 
We say $p\in\R[z_1,\dots,z_n]$ is {\em completely log-concave/Lortenzian} if for any $k \geq 0$, and any set of vectors $a_1,\dots,a_k\in \R^n_{\geq 0}$, $D_{a_1}\dots D_{a_k}p$ is non-negative and log-concave over $\R^n_{\geq 0}$, where for a vector $a\in \R^n$, $D_{a}=\sum_i a_i \partial_{z_i}$ is the directional derivative operator.  Completely log-concave polynomials were introduced in \cite{AOV18}  and  extended in \cite{ALOV18ii,BH20}.  We say a polynomial $p\in \R[z_1,\dots,z_n]$ is {\em multiaffine} if every monomial of $p$ is square-free, i.e., $\supp(p)\subseteq \{0,1\}^n$. A polynomial $p\in \R[z_1,\dots,z_n]$  is {\em $d$-homogeneous} if $p(\alpha z)=\alpha^d p(z)$ for any $\alpha\in \R$. 
Given a non-negative function $f:2^{[n]}\to\R_{\geq 0}$, we say $f$ is a completely log-concave set function if the generating polynomial of $f$, i.e., $$p_{f}(z_1,\dots,z_n)=\sum_{S \subseteq[n]} f(S) \cdot \prod_{i \in S} z_{i}$$   
is completely log-concave. Note that  $p_f$ is a multiaffine polynomial.  We define $\deg (f)$ to be the degree of polynomial $p_f$. 

 Completely log-concave polynomials have  nice properties that make them useful tools for  design and analysis of algorithms and studying mathematical objects like matroids. One of these properties is that, as first shown in \cite{Gurvits10}, if the polynomial $\sum_{i=0}^{n} c_i y^{n- i} x^i $ is completely log-concave,  then $c_0, \dots, c_n$ is an ultra log-concave sequence. 
 This property is used in \cite{ALOV18} to prove Mason's ultra log-concavity conjecture for independent sets of matroids. Another useful property is that, as first shown in \cite{ALOV18ii} and later improved in  \cite{CGM19, ALOVV21},  given a $d$-homogeneous log-concave set function, a natural random walk can be used to rapidly sample a subset $S$ of the ground set with probability proportional to $f(S)$. Note that any homogeneous log-concave set function, or equivalently any multiaffine homogeneous polynomial, is completely log-concave \cite{ALOV18}.   

An interesting aspect of  homogeneous completely log-concave polynomials  is  that the family of sets that can serve as the support of these polynomials can be nicely characterized. Given a polynomial $p=\sum_{\alpha\in \Z_{\geq 0}^n} c_\alpha z^\alpha$, the Newton polytope of $p$ is defined as
$$ \text{supp}(p) = \{\alpha\in \Z_{\geq 0}^n: c_\alpha\neq 0\},\quad \text{Newt}(p)=\text{conv}(\text{supp}(p)).$$
The following theorem gives a nice characterization of the supports of homogeneous completely log-concave polynomials. 
\begin{theorem}[\cite{ALOV18,BH20}]
Given a homogeneous completely log-concave polynomial $p\in \R_{\geq 0}[z_1,\dots,z_n]$, the Newton polytope of $p$ is a  generalized permutahedron,  namely it is a polytope all of whose edges are parallel to $\bone_i-\bone_j$ for $1\leq i<j\leq j$ and every integer point in this polytope also belongs to the $\text{supp}(p)$. Conversely, for any generalized permutahedran, there is a homogeneous completely log-concave polynomial with support equal to all integer points in the polytope.
\end{theorem}
 An immediate consequence of the above theorem is that if $p$ is homogeneous multiaffine and  log-concave then $\text{Newt}(p)$ is the base polytope of a matroid $M$ with ground set of elements $[n]$ (see \cite{GGMS87}). Having the above theorem that characterizes  supports of homogeneous log-concave polynomials, a natural question is whether one can give a more fine characterization of the set of possible coefficients of homogeneous log-concave polynomials. While such characterizations are not known, there are a number of results that give interesting necessary conditions for coefficients of  log-concave polynomials.  One such condition is implied by the following lemma. 
\begin{lemma} [\cite{ALOV18}]\label{lem:almost-log-submodular}
A polynomial $h(y,z ) = a + by+cz+ dyz \in \mathbb{R}[y,z]$ with non-negative coefficients is log-concave if  and only if $2bc \geq ad$. 
\end{lemma} 
\begin{corollary}[\cite{BH20}]\label{cor:logsubmodular}
If $p_f$ is  log-concave  for a non-negative set  function $f\in 2^{[n]} \rightarrow \mathbb{R}_{\geq 0}$, then $f$ is almost log-submodular, i.e. for  any $S\subset [n]$ and $i, j \in [n] \setminus S$,  $$2 f(S \cup \{i\}) f(S \cup \{i\}) \geq f(S) f(S \cup \{i, j\}).$$ 
\end{corollary}
\begin{proof}
To see this, note that the set of multiaffine log concave polynomials is closed under differentiation and specialization (see \cref{prop:log-concave_closure}). In particular, the following polynomial is log-concave: 
\begin{align*}
q (z_i, z_j) & = \left( (\prod_{k \in S} \partial_{z_k}) p_f (z_1, \dots, z_n)\right) _{| \{z_\ell =0\}_{ \ell \in [n] \setminus (S \cup \{i,j\})}} \\&= f(S) + f(S \cup \{i\}) z_i +  f(S \cup \{j\}) z_j +  f(S \cup \{i, j\}) z_i z_j.
\end{align*}
Thus, by  \cref{lem:almost-log-submodular}, $f$ is almost submodular. 
\end{proof}
Given a polynomial $p \in \R[z_1,\dots,z_n]$, we define the homogenization of $p$ as $\Hom(p, y) \coloneqq \sum_{i=0}^{n} y^{n+1- i} p_i $, where $p_i$ is the $i$-homogeneous part of $p$, i.e.  $p = p_0 + \dots + p_d$ and  $p_i$ is a $i$-homogeneous polynomial. Then, \cref{cor:logsubmodular} implies that  given a non-negative set  function $f\in 2^{[n]} \rightarrow \mathbb{R}_{\geq 0}$, if $\Hom (p_f, y)$ is completely log-concave, then $f$ is almost log-submodular. Thus, to take a step toward finding a classification of coefficients of homogeneous completely log-concave polynomials, a natural question to ask is whether one can find a large subclass of non-negative almost log-submodular functions such that for every $f$ in that subclass, $\Hom (p_f, y)$  is completely log-concave. Note that if $\Hom (p_f, y)$ is completely log-concave, then all homogeneous parts of $p_f$ are also log-concave.
 An important subclass of non-negative log-submodular functions with numerous applications is the class of non-negative monotone submodular functions (see \cref{fact:submodular-to-logsub}). Note that, however,   even within  this subclass,  one can find functions that are not  log-concave, and thus their homogenizations are not log-concave either (see \cref{cor:main-neg} for an example). 



\paragraph{Our Contribution.} 

In this paper, we introduce an expressive subclass of  non-negative monotone   functions, called strongly 2-coverage functions, and  prove  that $\Hom (p_f, y)$ is completely log-concave. We show that the set of strongly $2$-coverage functions  includes several  fundamental classes of non-negative monotone submodular functions including  matroid rank functions,  coverage functions, and, more generally, matroid rank sum functions, which are positive linear combination of rank functions and include a large subset of submodular functions that have been studied in the mechanism design literature \cite{CCPV7, DRY11, D11, DV11}.  As a consequence we prove the following theorem. 
\begin{theorem} If $f:2^{[n]}\to\R_{\geq 0}$ is a coverage function or a sum of matroid rank functions, then the sequence $f_0, f_1,\dots, f_n$ is ultra log-concave where $f_i=\sum_{S:|S|=i} f(S)$.
\end{theorem}
Moreover, we introduce a strictly larger class of non-negative submodular functions called 2-coverage functions, which, for instance, also include the indicator function of independent sets of any matroid. We  prove that if $p_f$ is 2-coverage,  all homogeneous parts of $p_f$  are  log-concave. As a consequence, given a 2-coverage function, one can use the results of  \cite{ALOV18ii, ALOVT21, CGM19}  to rapidly sample a subset $S$ of $[n]$ with probability proportional to $f(S)$.

\begin{figure}[!h]
\centering

\tikzset{every picture/.style={line width=0.75pt}} 

\begin{tikzpicture}[x=0.75pt,y=0.75pt,yscale=-1,xscale=1]

\draw   (275,210) -- (425,210) -- (425,260) -- (275,260) -- cycle ;
\draw   (275,120) -- (425,120) -- (425,170) -- (275,170) -- cycle ;
\draw   (465,120) -- (615,120) -- (615,170) -- (465,170) -- cycle ;
\draw   (85,90) -- (235,90) -- (235,140) -- (85,140) -- cycle ;
\draw    (350,172) -- (350,204.5) ;
\draw [shift={(350,206.5)}, rotate = 270] [color={rgb, 255:red, 0; green, 0; blue, 0 }  ][line width=0.75]    (10.93,-3.29) .. controls (6.95,-1.4) and (3.31,-0.3) .. (0,0) .. controls (3.31,0.3) and (6.95,1.4) .. (10.93,3.29)   ;
\draw   (85,150) -- (235,150) -- (235,200) -- (85,200) -- cycle ;
\draw    (237,115) -- (268.32,135.41) ;
\draw [shift={(270,136.5)}, rotate = 213.08] [color={rgb, 255:red, 0; green, 0; blue, 0 }  ][line width=0.75]    (10.93,-3.29) .. controls (6.95,-1.4) and (3.31,-0.3) .. (0,0) .. controls (3.31,0.3) and (6.95,1.4) .. (10.93,3.29)   ;
\draw   (85,210) -- (235,210) -- (235,260) -- (85,260) -- cycle ;
\draw    (237,235) -- (270,235.47) ;
\draw [shift={(272,235.5)}, rotate = 180.82] [color={rgb, 255:red, 0; green, 0; blue, 0 }  ][line width=0.75]    (10.93,-3.29) .. controls (6.95,-1.4) and (3.31,-0.3) .. (0,0) .. controls (3.31,0.3) and (6.95,1.4) .. (10.93,3.29)   ;
\draw    (237,175) -- (268.32,154.59) ;
\draw [shift={(270,153.5)}, rotate = 146.92] [color={rgb, 255:red, 0; green, 0; blue, 0 }  ][line width=0.75]    (10.93,-3.29) .. controls (6.95,-1.4) and (3.31,-0.3) .. (0,0) .. controls (3.31,0.3) and (6.95,1.4) .. (10.93,3.29)   ;
\draw   (275,300) -- (425,300) -- (425,350) -- (275,350) -- cycle ;
\draw    (350,262) -- (350,293.5) ;
\draw [shift={(350,295.5)}, rotate = 270] [color={rgb, 255:red, 0; green, 0; blue, 0 }  ][line width=0.75]    (10.93,-3.29) .. controls (6.95,-1.4) and (3.31,-0.3) .. (0,0) .. controls (3.31,0.3) and (6.95,1.4) .. (10.93,3.29)   ;
\draw    (540,172) -- (540,204.5) ;
\draw [shift={(540,206.5)}, rotate = 270] [color={rgb, 255:red, 0; green, 0; blue, 0 }  ][line width=0.75]    (10.93,-3.29) .. controls (6.95,-1.4) and (3.31,-0.3) .. (0,0) .. controls (3.31,0.3) and (6.95,1.4) .. (10.93,3.29)   ;
\draw   (465,210) -- (615,210) -- (615,260) -- (465,260) -- cycle ;
\draw   (275,390) -- (425,390) -- (425,440) -- (275,440) -- cycle ;
\draw    (350,352) -- (350,383.5) ;
\draw [shift={(350,385.5)}, rotate = 270] [color={rgb, 255:red, 0; green, 0; blue, 0 }  ][line width=0.75]    (10.93,-3.29) .. controls (6.95,-1.4) and (3.31,-0.3) .. (0,0) .. controls (3.31,0.3) and (6.95,1.4) .. (10.93,3.29)   ;
\draw    (427,145) -- (460,145.47) ;
\draw [shift={(462,145.5)}, rotate = 180.82] [color={rgb, 255:red, 0; green, 0; blue, 0 }  ][line width=0.75]    (10.93,-3.29) .. controls (6.95,-1.4) and (3.31,-0.3) .. (0,0) .. controls (3.31,0.3) and (6.95,1.4) .. (10.93,3.29)   ;

\draw (350,415) node   [align=left] {\begin{minipage}[lt]{102pt}\setlength\topsep{0pt}
\begin{center}
efficient sampling
\end{center}

\end{minipage}};
\draw (160,235) node   [align=left] {\begin{minipage}[lt]{115.6pt}\setlength\topsep{0pt}
\begin{center}
independent set \\functions of matroids
\end{center}

\end{minipage}};
\draw (160,175) node   [align=left] {\begin{minipage}[lt]{115.6pt}\setlength\topsep{0pt}
\begin{center}
matroid \\rank functions
\end{center}

\end{minipage}};
\draw (160,115) node   [align=left] {\begin{minipage}[lt]{115.6pt}\setlength\topsep{0pt}
\begin{center}
coverage functions
\end{center}

\end{minipage}};
\draw (350,145) node   [align=left] {\begin{minipage}[lt]{115.6pt}\setlength\topsep{0pt}
\begin{center}
strongly 2-coverage
\end{center}

\end{minipage}};
\draw (540,145) node   [align=left] {\begin{minipage}[lt]{115.6pt}\setlength\topsep{0pt}
\begin{center}
CLC\\homogenization
\end{center}

\end{minipage}};
\draw (350,325) node   [align=left] {\begin{minipage}[lt]{115.6pt}\setlength\topsep{0pt}
\begin{center}
log-concave \\homogeneous parts 
\end{center}

\end{minipage}};
\draw (540,235) node   [align=left] {\begin{minipage}[lt]{115.6pt}\setlength\topsep{0pt}
\begin{center}
ultra log-concave\\sequence
\end{center}

\end{minipage}};
\draw (350,235) node   [align=left] {\begin{minipage}[lt]{115.6pt}\setlength\topsep{0pt}
\begin{center}
 2-coverage
\end{center}

\end{minipage}};

\end{tikzpicture}
  \caption{Summary of results}\label{mylabel}

\end{figure}
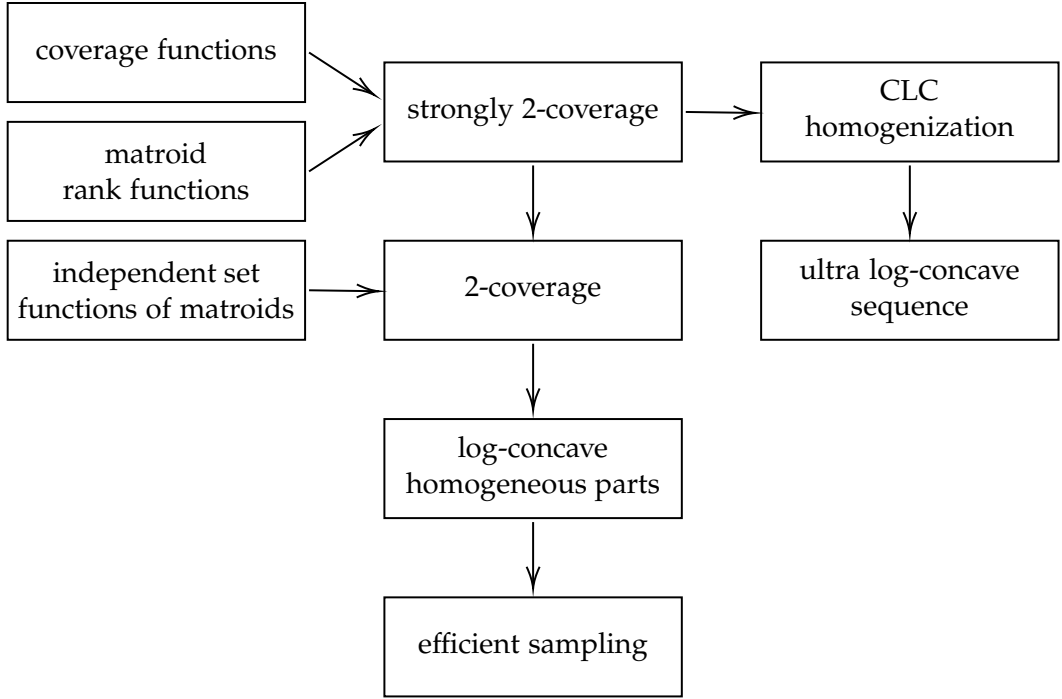






\section{Preliminaries}
Throughout the paper, we assume that for any set function  $f^{[n]} \rightarrow \mathbb{R}_{\geq 0}$   that we work with, we have $f(\emptyset) = 0$.  
We write  $[n]$ to denote $\{1,\dots, n\}$. 
For a function $f:2^{[n]}\to\R$ and an integer $0\leq d\leq n$ define the $d$-homogeneous restriction of $f$, $f^{(d)}:2^{[n]}\to\R$ as follows:
$$f^{(d)}(S)=\begin{cases} f(S) &\text{if }|S|=d, \\ 0 &\text{otherwise.}\end{cases}.$$
\begin{fact} \label{fact:crossterms}
For any $x, y \in \mathbb{R}$, $-2xy \leq c x^2 + \frac{1}{c} y^2$. 
\end{fact}
\subsection{Linear Algebra}

We write $J_n$ to denote the $n \times n$ all-ones matrix.  For any integer  $ n>0$ and  indices $i,j$,  define $E_{n, ij} \in \mathbb{R}^{n \times n}$ as  $E_{n, ij} (i, j)= 1$  and  let  every other entry to be zero. We drop the subscript  $n$ when the dimension is clear from context.
\begin{lemma}\label{lem:dj}
For any  diagonal matrix $D \succeq 0$, $JD+DJ$ has at most one positive eigenvalue.
\end{lemma}
\begin{proof}
We can write
\[JD+DJ = (D+I )J(D+I)  - DJD - J.\]
Now, the statement follows from the fact that  $(D+I )J(D+I)$, $J$, and $DJD$ are all rank-one PSD matrices.
\end{proof}

\begin{lemma}[\cite{ALOV19}] \label{lem:normalize_1positive}
Let $A \in \mathbb{R}^{n \times n}$ be a symmetric matrix. For any $P \in \mathbb{R}^{m \times n}$, if $A$ has at most one positive eigenvalue, $PAP^\top$ has at most one positive eigenvalue. 
\end{lemma}


\subsection{ Log Concave Polynomials} \label{sec:com_log_concave}

\begin{proposition} [\cite{BH20}; \cite{ALOV18}] \label{prop:log-concave_closure}
Given a polynomial $p \in \mathbb{R}[z_1, \dots, z_n]$
\begin{enumerate}[i]
\item If $p$ is (completely) log-concave, then $p_{|z_i =a}$ is also (completely) log-concave for  any $a \in \mathbb{R}_{\geq 0}$ and $1 \leq i \leq n$.  \label{item: specialization}
\item If $p$ is (completely) log-concave, any $c, \lambda_{1}, \ldots, \lambda_{n} \in \mathbb{R}_{\geq 0}$, $ c p(\lambda_1 z_1, \dots, \lambda_n z_n)$ is also (completely) log-concave. \label{item: scaling}
\item If $p$ is completely log-concave, then  $\partial_{z_i} p$  is also completely log-concave for any $1 \leq i \leq n$. Moreover, if $p$ is multiaffine and log-concave,  then  $\partial_{z_i} p$  is also  log-concave for any $1 \leq i \leq n$. \label{item: differentiation}
\end{enumerate}
\end{proposition}
\begin{proof}
\cref{item: specialization} and \cref{item: scaling}  follow from  definition. The first part of \cref{item: differentiation} also follows from the definition of log-concavity. We prove the second part without loss of generality for $\partial_{z_1} p$. Let $q_t (z_1, \dots, z_n) = t^ {-1} ( t z_1, \dots, z_n)$.  By \cref{item: scaling}, $q_t$ is log-concave for any $t \geq 0$. Note that since $p$ is multiaffine, we can write $ p = z_1 g+ h$, where $g, h \in \mathbb{R}[z_2, \dots, z_n]$, and hence $\partial_{z_1} p = g$. Assume that $g$ is not log-concave.   Therefore, for some $0 \leq \lambda \leq 1$ and $a, b \in \mathbb{R}_{\geq 0}$,  $g ( \lambda a +  (1- \lambda) b) >    g( a)^\lambda +  g(b)^{1- \lambda} $.  Note that $q_t  =  z_1 g+ t^{-1} h$. So for sufficiently large $t \in \mathbb{R}_{\geq 0}$,  $q_t ( \lambda a +  (1- \lambda) b) >   \lambda q_t( a) +  (1- \lambda) q_t(b) $, which is a contradiction with the fact that $q_t$ is log-concave. 
\end{proof}

For $d$-homogeneous polynomials, the following lemma gives an equivalent condition to log-concavity that in many cases is much easier to verify. 
\begin{lemma} [\cite{AOV18}]\label{log-concave-lemma}
Let $p \in \mathbb{R}_{\geq 0}\left[x_{1}, \ldots, x_{n}\right]$ be a $d$-homogeneous polynomial for some integer $d \geq 0$. For $a \in \mathbb{R}^n_{\geq 0}$ such that $p(a) \neq 0$, $p$ is log-concave at $a$ if and only if, $\nabla^2 p(a)$ has at most one positive eigenvalue. 
\end{lemma}

A polynomial $p\in \R[z_1,\dots,z_n]$ is {\em decomposable} if it can be written as a sum of two nonzero polynomials $f$ and $g$ such that $f$ and $g$ are supported on disjoint sets of variables. We say $p$ is {\em indecomposable} otherwise.  For a vector $\alpha \in \mathbb{Z}_{\geq 0}^n$ and  a   polynomial $p\in\R[z_1,\dots,z_n]$, define $ \partial^\alpha p \in \R[\{z_i\}_{i\notin\tau}]$  as
$$ \partial^\alpha p:=(\prod_{i=1}^n \partial^{\alpha_i}_{z_i}) p.$$ 
When $p\in\R[z_1,\dots,z_n]$ is a multiaffine polynomial, we might represent partial derivatives by subsets $\tau \subseteq [n]$, i.e. $\partial^\tau p:=(\prod_{i\in \tau} \partial_{z_i}) p$. In this case, we sometimes write $p_\tau$ to denote $\partial^\tau p$.
 The following theorem by \cite{ALOV18}, provides a very useful tool for proving a $d$-homogeneous polynomial is completely log-concave.

\begin{theorem}[\cite{ALOV18}] \label{log-concave-thm}
Let $p \in \mathbb{R}_{\geq 0}\left[x_{1}, \ldots, x_{n}\right]$ be a $d$-homogeneous polynomial. Then, $p$ is completely log-concave  if
\begin{enumerate}[i)]\item For any $\alpha \in \mathbb{Z}_{\geq 0}^n$, of size $\|\alpha \|_1\leq d-2$, $ \partial^\alpha p $ is indecomposable.
\item For any $\alpha\in \mathbb{Z}_{\geq 0}^n$, of size $\|\alpha \|_1 =  d-2$, the quadratic polynomial $ \partial^\alpha p $ is log-concave.
\end{enumerate}
\end{theorem}


The following proposition  is a slight modification of a statement by Gurvits \cite{Gurvits10}. For completeness, we include a short proof for this slightly modified version. 
\begin{proposition}\label{prop:ULC}
If $p=\sum_{k=0}^{n} c_{k} y^{n-k+1} z^{k} \in \mathbb{R}[y, z]$ is completely log- concave, then the sequence $c_0, . . . , c_n$ satisfies the following: for any $1 < k < n$,
\[
\left( \frac{c_k}{{n +1\choose k} }\right)^2 \geq  \left(\frac{c_{k-1}}{{n +1\choose k-1} }\right)\left(\frac{c_{k+1}}{{n +1\choose k+1} }\right). 
\]
We call such sequences ultra log-concave. 
\end{proposition}
\begin{proof}
Fix $1 <k<n$. Define $q(y, z) \coloneqq \partial_{y}^{n-k} \partial_{z}^{k-1} p$.  Using the fact that $\partial_{y}^{n+1-m} \partial_{z}^{m} p=(n+1-m) ! m ! c_{m}=(n+1) ! \frac{c_{m}}{{n+1 \choose m}}$  for any $0 \leq m \leq n+1$,  we compute $\nabla^{2} q$ as follows
\[
\nabla^{2} q=\left[\begin{array}{cc}\partial_{y}^{2} q & \partial_{y} \partial_{z} q \\ \partial_{y} \partial_{z} q & \partial_{z}^{2} q\end{array}\right]=   (n+1) !\left[\begin{array}{cc} \frac{c_{k-1}}{ {n+1 \choose  k-1}} & \frac{c_k}{{n+1 \choose  k} }\\  \frac{c_k}{{n+1 \choose  k-1}} &\frac{c_{k+1}}{ {n+1 \choose  k+1} }\end{array}\right]. 
\]
Since $p$ is completely log-concave, by \cref{prop:log-concave_closure}, $q$ is log-concave. By  \cref{log-concave-lemma},  the $2 \times 2$ matrix $\nabla^2 q$ has at most one positive eigenvalue. Since all entries of $\nabla^2 q$ are positive, the matrix has exactly one positive eigenvalue (and one negative eigenvalue), therefore its  determinant is non-positive. Therefore, we have
\[
\left( \frac{c_k}{{n +1\choose k} }\right)^2 -   \left(\frac{c_{k-1}}{{n +1\choose k-1} }\right)\left(\frac{c_{k+1}}{{n +1\choose k+1} }\right) \leq 0,
\]
as desired. 
\end{proof}

Note that, our definition of ultra log-concavity is slightly different from  the more common definition which says that a sequence  $c_0, \dots , c_n$  is ultra log-concave if it  satisfies $\left( \frac{c_k}{{n \choose k} }\right)^2 \geq  \left(\frac{c_{k-1}}{{n \choose k-1} }\right)\left(\frac{c_{k+1}}{{n\choose k+1} }\right)$. 
 
For any set function $f:2^{[n]}\to\R_{\geq 0}$ and $\tau\subseteq [n]$, define $f_\tau$ as the  (non-negative) function with generating polynomial $(p_{f})_{\tau}$, where $p_{f}$ is the generating polynomial of $f$. We say $f$ is $d$-homogeneous (resp. indecomposable) if $p_f$ is $d$-homogeneous (resp. indecomposable).

\subsection{Random Walks and  Completely Log-Concave Set Functions}

The mixing of a Markov chain quantifies the rate of its convergence to its stationarity distribution $\mu$. Given the transition probability matrix $P$ of a random walk on the state space $\Omega$, we define the mixing time as follows. For an initial state  $S_0 \in \Omega$ in  and an error parameter $\epsilon > 0$, define
\begin{align*}
    t_{\mix}(P, S_0, \epsilon) \coloneqq \min\{t \geq 0 : \norm{ P^{t} (S_0, \cdot ) - \pi}_{\TV} \leq \epsilon\},
\end{align*}
 where $\norm{\mu - \nu}_{\TV} = \frac{1}{2} \sum_{S \in \Omega} |\mu(S) - \nu(S)|$ gives the total variation distance between two distributions $\mu,\nu$ on $\Omega$.
Now, we define the  down-up walk on the support of a measure $\mu:2^{[n]} \rightarrow \mathbb{R}_{\geq0}$  as follows. The state space of the walk is all $S \subseteq [n]$ such that $|S| =d$ and $\mu(S) \neq 0$. At each step of the walk,  we drop an element $i \in S$ uniformly at random, then we go to  $j \notin S\setminus \{i\}$ with probability  $\frac{\mu(S\setminus \{i\} \cup \{j\})}{\sum_{k \notin S\setminus \{i\} } \mu(S\setminus \{i\} \cup \{k\})}$. Note that this markov chain is reversible and its stationary distribution is $\mu$, i.e.  $\mu(S)P(S, T) = \mu(T)P(T, S)$, where $P$ is the transition probability matrix of the down-up walk.  It was first  shown  by \cite{ALOV18ii} that  if $\mu:2^{[n]} \rightarrow \mathbb{R}_{\geq0}$  is log-concave, then the down-up walk mixes rapidly. 
Later, tighter mixing time results were shown by \cite{CGM19, ALOVV21}.
\begin{theorem} \cite{ALOVV21} \label{thm:logconcave_sampling}
Let $\mu: 2^{[n]} \rightarrow \mathbb{R} \geqslant 0$ be a  log-concave $d$-homogeneous set function. Then, starting at any state $S_0$, the mixing time of the down-up  walk on the support of this polynomial is 
\[
t_{\mix}\left(P, S_{0}, \epsilon\right) \leq O(d \log (d / \epsilon)). 
\]
\end{theorem}

\subsection{Submodular Functions}

A function $2^{[n]}\to\R$ is {\em submodular} if is has the diminishing return property, i.e.
\[\forall  S \subseteq T \subseteq [n], i \in [n]: f(S \cup \{i\})-f(S) \geq f(T \cup \{i\})-f(T).\]
Such a  function is {\em monotone} if $f(A) \leq f(B)$ for all $A \subseteq B$. 

\begin{fact}\label{fact:submodular-to-logsub}
Any non-negative monotone submodular function $f:2^{[n]}\to\R$ is log-submodular, i.e. $\log f$ is submodular.
\end{fact}
\begin{proof}
Fix set $S \subseteq T \subseteq [n]$. It is enough to show that $$\log \frac{f( S\cup \{i\})}{ f( S)} =  \log f( S\cup \{i\})  - \log f( S) \geq \log f( T\cup \{i\})  - \log f( T) = \log \frac{f( T\cup \{i\})}{ f( T)}.$$
But one can easily verify that  the diminishing return property combined with the fact that $f$ is non-negative and monotone implies  $\frac{f( S\cup \{i\})}{ f( S)}  \geq \frac{f( T\cup \{i\})}{ f( T)}$. Since $\log$ is an increasing function, this finishes the proof.
\end{proof}
A fundamental class of non-negative monotone submodular functions are coverage functions. 
\begin{definition}[Coverage Functions] Given a finite universe $U$ and sets $A_1, \dots, A_n \subseteq U$ and a measure  $w$ on  $U$,  we define $f: 2^{[n]} \rightarrow \mathbb{R}_{\geq 0}$ to be
\[
\forall T \subseteq[n]: f(T)=w \left(\bigcup_{i \in T} A_{i}\right).
\] 
\end{definition}
Linear set functions are a subclass of coverage functions.
\begin{definition}[Linear Set Functions] $f: 2^{[n]} \rightarrow \mathbb{R}_{\geq 0}$ is a linear function if 
\[
\forall T \subseteq[n]: f(T)=\sum_{i \in T} f(\{i\}). 
\]
\end{definition}

We will use the  following proposition which characterizes coverage set functions. 

\begin{proposition} \cite{CH12} \label{prop:coverage_condition}
A function $f: 2^{[n]} \rightarrow \mathbb{R}_{\geq 0}$ is a coverage function if and only if there exists a non-negative real number  $x_T$ for any $T \subseteq [n]$, such that  for any $S \subseteq [n]$ we can write  $f(S) = \sum_{T: T\cap S \neq \emptyset} x_T$. 
\end{proposition}

Many interesting classes of nonnegative monotone submodular functions are special cases of coverage functions. For example,  given a set of random variables $\Omega = \{ Y_1, \dots, Y_n\}$, the joint entropy function defines a coverage function on $2^{[n]}$. 

\begin{proposition}\label{prop:EntropyisCoverage}
 Given random variables $\Omega = \{ Y_1, \dots, Y_n\}$, let $H: 2^{\Omega } \rightarrow \mathbb{R}_{\geq 0}$ be the joint entropy function of these variables. 
Define  $f: 2^{[n]} \rightarrow \mathbb{R}_{\geq 0}$  as $f(S) = H(Y_S)$ for any $S\subseteq [n]$, where $Y_S:= \{Y_i| i \in S\}$. Then, $f$ is a coverage function. 
 \end{proposition}
 A proof of this proposition can be found in \cref{proof:entropy-coverage}
 
\par Coverage functions  are  themselves special cases of matroid rank functions, as any coverage function can be written as a weighted sum of rank 1 matroids each corresponding to the coverage of an individual element. 
\begin{fact}\label{fact:coverage_rank}
Any coverage function can be written as positive sum of matroid rank functions.
\end{fact}


%% file: technical.tex
\section{Main Contributions}

We define the following classes of set functions. 
\begin{definition}[2-Coverage Set Functions]\label{def:2cov}
 A set function  $f: 2^{[n]} \rightarrow \mathbb{R}_{\geq 0}$ is $2$-coverage with respect to an integer $2\leq d\leq \deg (f) $ if 
 the following conditions holds. 
 \begin{enumerate}[label=(\roman*)]
 \item
For any $\tau \subseteq [n]$ with $|\tau| \leq  d-2$, $(f^{(d)})_\tau$ is indecomposable. 
\item
For all $\tau \subseteq [n]$ with $|\tau| = d-2$, there exists $S \subseteq [n]\setminus \tau$\footnote{Intuitively $S$ corresponds to the non-loop elements of the "contracted" version of $f$.},  a coverage function $g: 2^{S} \rightarrow \mathbb{R}_{\geq 0} $, and a linear set function $\ell: 2^{S} \rightarrow \mathbb{R}_{\geq 0}$ (that are possibly dependent on $\tau$), 
such that 
\begin{itemize}\item  $\ell (\{i\}) \leq  g(\{i\})$ for all $i \in S$
\item For any $T\subseteq [n]\setminus \tau$ of size $|T|=2$,
$$ f_\tau(T)=\begin{cases} 0 &\text{if } T\not\subset S,\\g(T)-\frac12 \ell(T),&\text{otherwise}\end{cases}$$
\end{itemize}

We say $f:2^{[n]}\to\R_{\geq 0}$ is 2-coverage if it is 2-coverage with respect to any $2\leq d \leq \deg (f) $. 
\end{enumerate}
\end{definition}

\begin{definition}[Strongly $2$-Coverage Set Functions]\label{def:strong2cov}
 A  set function $f: 2^{[n]} \rightarrow \R_{\geq 0}$ is strongly $2$-coverage if the following holds.
For all $\tau \subseteq [n]$ such that $0 \leq |\tau| \leq n-2$,  there exists a {\em coverage} function $g: 2^{[n] \setminus \tau} \rightarrow \mathbb{R}$,  such that   $(f_\tau)^{(1)}=g^{(1)} + f(\tau) $, and  $(f_\tau)^{(2)}=g^{(2)}  + f(\tau)$. 
\end{definition}

The following propositions capture basic properties of  these classes of set functions. A proof of these lemmas  are included in section \cref{sec:prop_and_examples}.   

\begin{proposition}\label{prop:strong_is_monotone_submodular}
Let $f: 2^{[n]} \rightarrow \mathbb{R}_{\geq 0}$ be strongly 2-coverage,  then  $f$ is monotone and submodular. 
\end{proposition}

\begin{proposition}\label{prop:strong_to_weak}
 Any strongly 2-coverage function $f: 2^{[n]} \rightarrow \mathbb{R}_{\geq 0} $ is  $2$-coverage.\end{proposition}
 
 \begin{proposition}\label{prop:convex_cone}
The set of  strongly 2-coverage functions on $[n]$ is a convex cone, i.e., if $f_1, f_2:2^{[n]}\to\R$	are  strongly 2-coverage  set functions,  then for any $\alpha\geq 0$, $\alpha f_1$ and $f_1+f_2$ are strongly 2-coverage functions.
\end{proposition}

The following proposition provides some examples of  strongly 2-coverage and 2-coverage functions.  A proof of this proposition  is included in section \cref{sec:prop_and_examples}.

 \begin{proposition} \label{prop:2coverage_functions}
\begin{enumerate}[label=(\roman*)]
\item For any matroid  $M = ([n], I)$,  its rank function  $ \operatorname{  rk}_M: 2^{[n]} \rightarrow \mathbb{R}$ is a strongly 2-coverage set function.  \label{2cov_part1}
\item For any matroid  $M = ([n], I)$, the indicator function of its independent sets is $2$-coverage. Combined with \cref{main}, this gives another  proof for the complete log-concavity of bases generating polynomial of a matroid that is proved in \cite{AOV18}. \label{2cov_part2}
\end{enumerate}
\end{proposition}
 
A consequence of  \cref{prop:2coverage_functions} and \cref{prop:convex_cone} is that  matroid rank sum functions are strongly 2-coverage. Therefore, by \cref{fact:coverage_rank} and \cref{prop:EntropyisCoverage},  the joint entropy functions and coverage functions are strongly 2-coverage.

\par One of our main results is the following theorem. A proof of this theorem is included in \cref{sec:2-cov_CLS}. 
\begin{theorem}\label{main}
If  $f : 2^{[n]} \rightarrow \mathbb{R}_{\geq 0} $ is 2-coverage with respect to some $2 \leq d \leq \deg(f)$,  then $f^{(d)}$ is completely log-concave. 
\end{theorem}

One of the consequences of \cref{main} is that for any   $f : 2^{[n]} \rightarrow \mathbb{R}_{\geq 0}$ that is 2-coverage with respect to  some $2 \leq  d \leq \deg (f)$,  one can sample a set $S \subseteq [n]$ of size $d$ proportionate to $f(S)$ in polynomial time. 

\begin{corollary}
Given a  set function $f : 2^{[n]} \rightarrow \mathbb{R}_{\geq 0}$  that is 2-coverage with respect to some $2 \leq d \leq \deg(f)$, let $\mu_d$ be the distribution induced by $f^{(d)}$, i.e. $\mu_d(S) =  \frac{f(S)}{\sum_{S \subset [n], |S| = d} f(S)}$, for any set $S$ of size $d$. For any $\epsilon >0$, starting from an arbitrary set $S_0$,  the up-down walk $P$ on sets of size $d$ generates a sample from $\hat{\mu}_d$ such that $\|\hat{\mu}_d - \mu_d \|_{\TV} \leq \epsilon$ in time  $O(d \log (d/\epsilon))$, i.e., 
\[t_{\text{mix}}(P, S_0, \epsilon) \leq O(d \log (d/\epsilon)). \]
\end{corollary}

The following theorem is our other main result. This theorem is proved in \cref{sec:strongly_2-cov_CLS}. 

\begin{theorem}\label{thm:inhomoglogconcavity}
Let $f: 2^{[n]} \rightarrow \mathbb{R_{\geq 0}}$ be  a strongly $2$-coverage set function. Then, the polynomial $ q_f (y, x_1, \dots, x_n) \coloneqq \sum_{i=0}^n y^ {n+1 -i} \sum_{S \subseteq [n], |S| = i} f(S) x^S$ is completely log-concave. 
\end{theorem}

The following corollary simply follows by an application of \cref{prop:ULC}.
\begin{corollary}
Let $f: 2^{[n]} \rightarrow \mathbb{R_{\geq 0}}$ be  a strongly $2$-coverage set function. Let $c_k$ be the $k$-th coefficient of  $ p_f (y, x) = \sum_{i=0}^n (\sum_{S \subseteq [n], |S| = i} f(S) )y^ {n+1 -i} x^i$. Then, for $1<k<n$, we get
\[
\left( \frac{c_k}{{n+1\choose k} }\right)^2 \geq  \left(\frac{c_{k-1}}{{n +1\choose k-1} }\right)\left(\frac{c_{k+1}}{{n +1\choose k+1} }\right)
\]
\end{corollary}

\section{Basic Properties of 2-Coverage and Strongly 2-Coverage  Functions} \label{sec:prop_and_examples}

\begin{proof}[Proof of \cref{prop:strong_is_monotone_submodular}]
Fix  $A \subseteq [n]$. We show that for any $ B \supseteq A$, 
 $f(A) \leq f(B)$ and for all $i \in [n]$, 
 $f(A \cup \{i\}) - f(A) \geq  f(B \cup \{i\}) - f(B)$. We prove by induction on $|B \setminus A|$. The claims trivially hold for $|B \setminus A| = 0$. Now, let  $ |B \setminus A|  = \ell$ for some $1 \leq  \ell \leq n - |A|$. Take some $j \in B \setminus A$ and let $C \coloneqq B \setminus \{j\}$. By induction hypothesis,  $f(A) \leq f(C)$. Furthermore, by definition there a coverage function $g: 2^{[n] \setminus C} \rightarrow \mathbb{R}_{\geq 0}$ such that  $f(B) = f_C(\{j\}) = g(\{j\}) + f(C)$. Therefore, $f(A) \leq f(C) \leq f(B)$. To prove the submodularity condition, note that by induction hypothesis, we have $f(A \cup \{i\}) - f(A) \geq  f(C \cup \{i\}) - f(C)$ for all $i \in [n]$. Furthermore, by submodularity of $g$, we have
  $$  f(C \cup \{i\}) - f(C) = g(\{i\}) \geq g (\{j, i\}) - g(\{j\}) =  f( C \cup \{i,j\}) - f( C \cup \{j\})  = f(B \cup \{i\}) -  f(B).  $$
Therefore, $f(A \cup \{i\}) - f(A) \geq  f(B \cup \{i\}) - f(B)$ as desired. 
\end{proof}

\begin{proof}[Proof of \cref{prop:strong_to_weak}]
It is enough to show that given for any $2 \leq d \leq \deg (f)$, and any $\tau \subseteq [n]$ with $|\tau| \leq d-2$, $(f^{(d)})_\tau$ is indecomposable.  If $(f^{(d)})_\tau = 0$ this claim follows trivially.  Assume $(f^{(d)})_\tau \neq 0$ and let $p \coloneqq p_{f^{(d)}}$ be the generating polynomial of $f^{(d)}$.   For the sake of contradiction, assume that $p_\tau = h+g$ such that $h, g$ are both non-zero. Let $h$ and $g$ be respectively supported on   disjoint set of variables $S_h$ and $S_g$. By \cref{prop:strong_is_monotone_submodular}, we know that $f$ is monotone. Therefore, for any $x_i \in S_h$,  $x_j \in S_g$, we must have 
\[f(\tau \cup \{i\})  \leq f(\tau \cup \{i,j\}) =   f_\tau(\{i,j\}) =0. \]
This implies that $f_\tau ( \{i\}) = f(\tau \cup \{i\})  =0$. 
Similarly, $f_\tau(\{j\}) = 0$. Therefore, $f_\tau$ is $0$ on all sets of size 1.   By monotonicity of $f$, we get $f_\tau = 0$ which contradicts our assumption. 
\end{proof}

\begin{proof}[Proof of \cref{prop:convex_cone}]
Fix a $\tau \in [n]$ such that $0 \leq |\tau| \leq n-2$. Let $g_1, g_2: 2^{[n] \setminus \tau} \rightarrow \mathbb{R}_{\geq 0}$ be coverage functions such that $((f_1)_\tau)^{(2)}= (g_1)^{(2)} $ and $((f_2)_\tau)^{(2)}= (g_2)^{(2)}$. 
The statement follows from the fact that if  $g_1, g_2$ are coverage functions, then  $\alpha g$ and $g_1+g_2$ are also coverage functions. 
\end{proof}

\begin{proof}[Proof of \cref{prop:2coverage_functions}]

We prove \cref{2cov_part1}. Let $f \coloneqq \operatorname{  rk}_M$. For any $\tau \subseteq [n]$ with $|\tau| \leq n-2$, $f_\tau (S) = \operatorname{rk}_M(\tau) + \operatorname{rk}_{M/\tau}(S)$ for all  sets $S \subseteq [n] \setminus \tau$ of size 1 or 2. Note that $f(\tau) = \operatorname{rk}_M(\tau)$. To satisfy the condition of \cref{def:strong2cov}, it is enough to show that there is a coverage function $g$ that takes the same values as    $\operatorname{rk}_{M/\tau}(S)$ on sets of size $1$ and $2$. 
Using matroid partition property, $M/\tau$ can be partitioned into sets $S_0, \dots, S_k$ such that the following holds. 
\[
\operatorname{rk}_{M/\tau}(\{x, y\}) = 
     \begin{cases}
       \text{0} &\quad \text{if } x, y \in S_0\\
       \text{1} &\quad  \text{if } x\in S_0, y\in S_i \quad \text{for some } 1\leq i\leq k \\
       \text{1} &\quad \text{if } x, y \in S_i \quad \text{for some } 1\leq i\leq k\\
       \text{2} &\quad \text{if } x \in S_i, y \in S_j \quad \text{for some } 1\leq i< j \leq k\\ 
     \end{cases}
\]
Define $g$ as follows. 
For any $ i \in \{1, \dots, k\}$ and $x \in S_i$, let $A_x := \{ i\}$, and for any $x \in S_0$, let $A_x = \emptyset$.  For any $T \in [n] \setminus \tau$,  define  $g(T)  \coloneqq |\cup_{x\in T} A_x|$. One can easily check that  $(f_\tau)^{(1)}=g^{(1)} + f(\tau) $, and  $(f_\tau)^{(2)}=g^{(2)}  + f(\tau)$. 
\par
Now, we prove \cref{2cov_part2}. Let $M = ([n], I)$ be a matroid of rank $r$ and let
$f: 2^{[n]} \rightarrow \mathbb{R}$ be an indicator function of independent sets, i.e. $f(S) = \mathbf{1}_I (S)$. We want to show that $f$ is $2$-coverage. The indecomposability holds because of the exchange property of matroids. We verify the second condition. Fix $2 \leq d \leq r$ and $\tau \subseteq [n]$ with $|\tau| =  d-2$.  Let $S$ be the set of non-loop elements of $M/\tau$. Using the matroid partition property, $S$   can be partitioned into sets $S_1, \dots, S_k$  such that 
\[
(f_{\tau})^{(2)}(\{x, y\})  =
     \begin{cases}
      
       \text{0} &\quad x, y \in S_i \quad \text{for some } 1\leq i\leq k\\
       \text{1} &\quad x \in S_i, y \in S_j \quad \text{for some } 1\leq i< j \leq k\\ 
     \end{cases}
\]
 For any $ 1 \leq i \leq k$ and $x \in S_i$, let $A_x := \{ i\}$.   Define  the coverage function $g: 2^S \rightarrow [n]$ as  $g(T)  = |\cup_{x \in T} A_x|$ for every $T \subseteq S$. Furthermore, consider the linear set function $\ell: 2^S \rightarrow [n]$  given by $\ell(\{i\})  = 1$ for all $i\in S$. It is easy to check that $f_\tau (T)= g (T) -\frac{\ell}{2} (T)$, for any  $T \subseteq S$ of size $2$. 
\end{proof}

\section{Complete Log-Concavity of  Homogeneous Parts} \label{sec:2-cov_CLS}
In this section, we prove that if a set function $f$ is  $2$-coverage with respect to some $2 \leq d \leq \deg(f)$, then $f^{(d)}$ is log-concave. 
\begin{lemma} \label{lem:mainPSD}
Let $g: 2^{[n]} \rightarrow \mathbb{R}_{\geq 0}$ be a coverage function and  $D$ be a diagonal matrix with $D(i,i) = g(\{i\})$ for all $i\in [n]$.   Then, $R \coloneqq  (DJ+JD)  - \nabla^2 p_{g^{(2)}} \succeq  D$, where $p_{g^{(2)}}$ is the generating polynomial of $g^{(2)}$. 
\end{lemma}
\begin{proof}
By \cref{prop:coverage_condition},  there exists non-negative $\{x_T\}_{T\subseteq [m]}$   such that  for any $i\neq j\in [n]$
\begin{align*}
g(\{i, j\}) &=\sum_{T: T \cap \{i,j\} \neq \emptyset} x_{T} =\sum_{T: i\in T} x_{T} + \sum_{T: j\in T} x_{T}  - \sum_{T: i,j\in T } x_T \\&=  g(\{i\}) +g(\{j\})  -  \sum_{T: i,j\in T } x_T. 
\end{align*}
Note that for all $i\neq j\in [n]$,  $(DJ + JD) (i,j) =  g(\{i\}) +g(\{j\}) $. 
Therefore, we can write $R$  as
$$R(i,j) =\begin{cases}  2 g(\{i\}) & \text{if } i=j
\\  \sum_{T: i,j\in T} x_T  & \text{otherwise.}\end{cases}$$
Furthermore,  for any $T \subseteq [m]$ define matrix $B_T$  as  $B_T (i, j) = 1$ if $\{i,j\} \subseteq T$ and $0$ otherwise. We can rewrite $R$ as follows 
\[
R =  \sum_{T \subseteq[m]} x_T B_T  + D. 
\]
For any $T$, $B_T\succeq 0$. Therefore,  since $x_T \geq 0$,  we have $\sum_{T \subseteq[m]} x_T B_T\succeq 0$. This implies that $R \succeq  D$,  as desired. 
\end{proof}

\begin{proof}[Proof of \cref{main}]
Let $p \coloneqq p_{f^{(d)}}$ be the generating polynomial of $f^{(d)}$. We use  \cref{log-concave-thm} to prove the theorem. 
Note that the indecomposability condition holds by definition.  Now, we use  \cref{log-concave-lemma}  to prove that the second condition of    \cref{log-concave-thm} holds.  Fix  $ \tau =\{ i_1, \dots, i_{d-2}\}$. We want to show that $p_\tau$ is log-concave.  If $ p_\tau$ is identically zero, the condition trivially holds. So, assume that $ p_\tau \neq 0$.  By ‌\cref{log-concave-lemma}, it is enough to show that $\nabla^2 ( p_\tau)$ has at most one positive eigenvalue.  By definition, there is a set $S \subseteq [n] \setminus \tau$, a coverage function  $g: 2^{[S]} \rightarrow \mathbb{R}_{\geq 0}$, and a linear set function $\ell: 2^{S} \rightarrow \mathbb{R}_{\geq 0}$ such that  for any $T \subseteq [n]\setminus \tau$ of size $2$,  $ f_\tau (T) =   g(T) - \frac{\ell(T)}{2}$ if $T \subseteq S$, and  $f_\tau (T) = 0$ otherwise. Therefore, it is enough to show that $Q =\nabla^2 p_{g^{(2)}} -\nabla^2 p_{\frac{\ell^{(2)}}{2}} $ has at most one positive eigenvalue, where $Q$ is the principle minor  of $\nabla^2 ( p_\tau)$ obtained by restricting it to the rows and columns indexed by $S$. We can write, $Q=\nabla^2 p_{g^{(2)}} - (\frac{JC+CJ}{2}-C)$ where  $C$ is a diagonal matrix with $C (i,i) = \ell (\{i\})$.  Furthermore,  let $D$ be a diagonal matrix with $D(i,i) = g(\{i\})$ for all $i\in S$. Note that by  \cref{lem:mainPSD}, we can write $ p_{g^{(2)}} + D\preceq DJ + JD$. So 
\begin{align*}
\nabla^2  (p_\tau)  &=\nabla^2 p_{g^{(2)}} +C - \frac{JC+CJ}{2}  \preceq  \nabla^2   p_{g^{(2)}} + D - \frac{JC+CJ}{2}\\ &\preceq  DJ+ JD  - \frac{JC+CJ}{2} =  (D- \frac{C}{2})J+ J (D- \frac{C}{2}). 
\end{align*} 
Thus, by \cref{lem:dj}, $\nabla^2  (p_\tau)$ has at most one positive eigenvalue. 
\end{proof}

\section{Complete Log-Concavity of the Homogenization of the Generating Polynomial of Strongly 2-Coverage Functions} \label{sec:strongly_2-cov_CLS}

\begin{proof}[Proof of \cref{thm:inhomoglogconcavity}]
We use  \cref{log-concave-thm}. First, we show that  $q_f$ is indecomposable.  We want to show that for any   $0 \leq k \leq n-1$, $0 \leq \ell \leq n-1-k$ and any $\tau = \{i_1 \dots, i_\ell\}$,  $ \partial^\tau \partial ^{k}_{y} q_f (y, x)$ is indecomposable (assuming it is non-zero). Note that,  for some polynomial $g(y, x)$, we can write $\partial^\tau \partial ^{k}_{y} q_f (y, x) = y g(y, x) + \partial^\tau p_{f^{(n-k+1)}}(x)$, where  $p_{f^{(n-k+1)}}$ is the generating polynomial of $f^{(n-k+1)}$. Now, by \cref{prop:strong_to_weak}, $\partial^\tau p_{f^{(n-k+1)}}$ is indecomposable. Furthermore, $y g(y, x)$ is indecomposable since $y$ appears in all  of its monomials.  It is enough to show that there exists a variable that appears in monomials of both  $yg(y, x)$ and   $\partial^\tau p_{f^{(n-k+1)}}$. 
If $\partial^\tau p_{f^{(n-k+1)}}$ is identically zero, we are done.
Otherwise, since $f_\tau$ is monotone and submodular by \cref{prop:strong_is_monotone_submodular}, there  exists $j$ such that $f_\tau(\{j\}) >  0$. Therefore, $x_j y^{n-\ell-k}$ is a monomial in $ yg(y, x)$. By monotonicity of  $f_\tau$, for any set $S \subseteq [n] \setminus \tau$ of size $n-k- \ell+1$ such that $j \in S$, $x^S$ is a monomial in $\partial^\tau p_{f^{(n-k+1)}}$. This finishes the proof of indecomposability. Now, we prove that the second condition  of   \cref{log-concave-thm} holds. It is enough to show that  for any  $0 \leq k \leq n-1$ and  $\tau = \{i_1 \dots, i_k\}$, $ \partial^\tau \partial_y^{n - 1 - k} q_f (y, x)$ is log-concave. Let $p = \frac{\partial^\tau \partial_y^{n- 1 - k} q_f (y, x)}{(n- 1-k)!}$. We can write 
\[
p  = \frac{(n-k+1)(n-k)}{2} f(\tau)  y^{2}+(n-k) \sum_{i  \in[n] \setminus \tau } f (\tau \cup \{i\}) y x_{i}+ \sum_{  \{i, j\}  \subseteq [n] \setminus \tau } f (\tau \cup \{i, j\}) x_{i} x_{j}.
\]
Without loss of generality, assume that $p$ is supported on variables $\{x_1, \dots, x_m\} \cup \{y\}$ for  $m \coloneqq n -k$.
We compute the Hessian matrix. In order to write the Hessian as a matrix indexed by $\{0, \dots, m\}$, we assume that $y$ corresponds to $0$ and for all $1 \leq i\leq m$,  $x_i$ corresponds to integer $i$. With this indexing, we can write the Hessian matrix as 

\[ H = 
\begin{bmatrix}
(m+1)m f(\tau) & m f_\tau(  \{1\})  & \dots & m f_\tau(  \{m\})\\
m f_\tau(  \{1\}) & 0 & \dots  &  f_\tau(  \{1, m \}) \\
\vdots &\vdots & \dots & \vdots  \\
m f_\tau(  \{m\}) & f_\tau(   \{m, 1 \})  & \dots  &  0\\
\end{bmatrix}. 
\]
By \cref{log-concave-lemma}, it is enough to show that $H$ has at most one positive eigenvalue. By \cref{lem:normalize_1positive},  this is equivalent to showing that the following matrix has at most one positive eigenvalue. 
\[ G= 
\begin{bmatrix}
\frac{m+1}{m} f(\tau) & g(  \{1\})+f(\tau)  & \dots &  g(  \{m\}) + f(\tau)\\
 g(  \{1\})+ f(\tau) & 0 & \dots  &  g(  \{1, m \})+f(\tau)  \\
\vdots &\vdots & \dots & \vdots  \\
 g(  \{m\}) +f(\tau) & g(   \{m, 1 \})+f(\tau)   & \dots  &  0
\end{bmatrix} \]
\[ = \begin{bmatrix}
0 & g(  \{1\}) & \dots &  g(  \{m\}) \\
 g(  \{1\}) & 0 & \dots  &  g(  \{1, m \}) \\
\vdots &\vdots & \dots & \vdots  \\
 g(  \{m\}) & g(   \{m, 1 \})  & \dots  &  0\\
\end{bmatrix} 
+ \begin{bmatrix}
\frac{m+1}{m} f(\tau) & f(\tau)  & \dots &  f(\tau)\\
 f(\tau) & 0 & \dots  &f(\tau)  \\
\vdots &\vdots & \dots & \vdots  \\
 f(\tau) & f(\tau)   & \dots  &  0
\end{bmatrix}.
\]
where the coverage function $g$ satisfies $f(S)=f(\tau)+g(S\setminus \tau)$ for sets $S$ such that  $|\tau|+1\leq |S|\leq |\tau|+2$ and  $\tau \subseteq S$. In the above line, let  $H, K$ denote the first and second matrix respectively.  Let $D \in R^{(m+1) \times (m+1)}$ be a diagonal matrix such that $D(0,0) = \frac{m+1}{2m} f(\tau)$ and $D(i,i) = \frac{m-1}{2m} f(\tau)$ for $1 \leq i \leq m$. Let $D' \in R^{(m+1) \times (m+1)}$ be a diagonal matrix with $D'(0,0) =0$ and $D'(i,i) = g(\{i\})$ for $1 \leq i \leq m$. Note that  by \cref{lem:dj}, $(D +D')J + J (D+D')$ has exactly one positive eigenvalue.
Therefore, to show that $H + K$ has at most one positive eigenvalue, it is enough to prove that $(D +D')J + J (D+D')- H- K \succeq 0 $. Note that $D'J + J D'- H \succeq 0$. To see this, first note that the first row and column of $ D'J + J D'- H$ are zero. 
Furthermore, one can  check that after eliminating the first row and column of $K$, the remaining matrix is equal to the Hessian of the generating polynomial of $g^{(2)}$. Thus, the claim follows from \cref{lem:mainPSD}. Therefore, it is enough to show that $D J + J D- K \succeq 0$. Again, it is easy to check that the first row and column of $D J + J D- K$ are zero. After removing the first row and column of $D J + J D- K$, the remaining matrix is equal to $f(\tau) ( I_m - \frac{1 }{m} J_{m})$.
Now, note that  $ J_{m} \preceq m I_m$, as $m$ is the the largest eigenvalue of $J_m$. Therefore, $D J + J D- K \succeq 0$, as desired. 


\end{proof}

\section{Negative Results} \label{sec:negativeres}
We showed that for an expressive class of nonnegative monotone submodular functions, the generating polynomial of $d$-homogenous restriction of the function    is completely log-concave for any $d \geq 1$. We show that this claim does not hold for all nonnegative monotone submodular  functions. 

\begin{proposition}
There exist integers $n \geq d \geq 1$ and  a non-negative monotone submodular function $f: 2^{[n]} \rightarrow \mathbb{R}_{\geq 0}$  such that $f^{(d)}$is not  log-concave. 
\end{proposition}
\begin{proof}
Let $[12]$ be the ground set. Furthermore, define $w_0 = \dots = w_5 = 1$, $w_6 = \dots = w_9 = 2$, and $w_{10} = w_{11} = 0$. Now, take the following set function $f: 2^{[12] }\rightarrow \mathbb{R}$ to be $f(S) =\min \{ \sum_{i \in S}  w_i, 2\}$. Note that this function is a budget additive function and is non-negative, monotone  and submodular. But, one can verify that the polynomial $p(x) = \sum_{\{i,j\} \subseteq [12] } f(\{i,j\}) x_i x_j$ is not log-concave as $\nabla^2 p(x)$ has two positive eigenvalues. 
\end{proof}
\begin{corollary}\label{cor:main-neg}
There exists monotone submodular function $f: 2^{[n]} \rightarrow \mathbb{R}_{\geq 0}$  such that $f$ is not  log-concave. 
\end{corollary}
\begin{proof}
This follows by the fact that if a  function $f$ is log-concave, all of its homogeneous parts are also log-concave. 
\end{proof}

Gelfand, Goresky, MacPherson, and Serganova proved that the support of any homogeneous multiaffine log-concave polynomial correspond to bases of a matroid \cite{GGMS87}. But there is not much known about the coefficient of these polynomials. A natural question to ask is that if the coefficients come from a   monotone submodular function that is non-negative on non-empty sets. Another natural question to consider is whether  the coefficients of these polynomials come from 2-coverage functions. The following proposition provides a counter-example to both of these  statements. 
\begin{proposition}
There exist integers $n \geq d \geq 1$ and a $d$-homogeneous multiaffine log-concave $p \in R_{\geq 0}[x_1, \dots, x_n]$ such that for any  $f: 2^{[n]}  \rightarrow \mathbb{R}_{\geq 0}$, if  the generating polynomial of $f^{(d)}$ is equal to $p$, $f$ is neither a monotone submodular function that is non-negative on non-empty sets nor a 2-coverage function. 
\end{proposition}
\begin{proof}
Let $p (x_1, x_2, x_3) := 3 x_1x_2 +   x_1x_3 + x_2x_3$. We have 
\[
\nabla^2 p = \begin{bmatrix}
0 & 3 & 1\\
3 & 0 & 1\\
1 & 1 & 0\\
\end{bmatrix}.
\]
One can easily check that $\nabla^2 p$ has exactly 1 positive eigenvalue. So, using \cref{log-concave-lemma}, $p$ is log-concave. Take an arbitrary $f: 2^{[3]}  \rightarrow \mathbb{R}_{\geq 0}$ such that the generating polynomial of $f^{(2)}$ is equal to $p$. To show the first part of the statement, we assume that $f$ is a   monotone  function that is non-negative on non-empty sets, and show that $f$ is not submodular.   By monotonicity, we have $f(\{1, 2, 3\})  \geq  f(\{1, 2\})$. Therefore
$$f(\{1, 2, 3\}) -  f(\{1, 3\}) \geq f(\{1, 2\}) -  f(\{1, 3\}) = 2.$$ Moreover,  $f(\{2, 3\})  -f(\{3\})  \leq  f(\{2, 3\}) \leq 1$. Therefore,  $f$ is not submodular. 

To show the second part of the statement, assume for contradiction that $f$ is a 2-coverage function.
 Since $f$ is non-zero on all sets of size 2, we must have $S=\{1,2,3\}$ (for $\tau=\emptyset$). 
So, there exists a coverage function $g: 2^{[3]} \rightarrow \mathbb{R}_{\geq 0}$ and a linear set  function $\ell: 2^{[3]} \rightarrow \mathbb{R}_{\geq 0}$ such that $f^{(2)} = (g-\frac{\ell}{2})^{(2)}$ and that $\ell(\{i\})\leq g(\{i\})$ for $i\in \{1,2,3\}$.
We show that $g$ is not submodular, so it cannot be a coverage function. Therefore
$$g(\{1, 2\}) -  g(\{1, 3\})  = f(\{1, 2\}) -  f(\{1, 3\}) + \frac{\ell(\{2\})  -\ell(\{3\})}{2}  = 2 + \frac{\ell(\{2\})  -\ell(\{3\})}{2}  .$$
 This implies that $g(\{1, 2, 3\}) -  g(\{1, 3\}) \geq 2 + \frac{\ell(\{2\})  -\ell(\{3\})}{2} $.  Moreover, 
\begin{align*}
g(\{2, 3\}) - g(\{3\}) &=  f (\{2, 3\}) + \frac{\ell (\{2\}) + \ell (\{3\})}{2} -  g(\{3\}) \\&= 1 - g(\{3\}) + \frac{\ell (\{2\}) + \ell (\{3\})}{2}  \underset{\ell(\{3\}\leq g(\{3\})}{\leq} 1+ \frac{\ell (\{2\}) - \ell (\{3\})}{2}. 
\end{align*}
Combining these, we get $g(\{1, 2, 3\}) -  g(\{1, 3\})  \geq g(\{2, 3\}) - g(\{3\})$, which is a contradiction with submodularity of $g$. 
\end{proof}

%% file: appendix.tex
\appendix 
\section{Entropy as a Coverage Function}
\begin{proof}[Proof of \cref{prop:EntropyisCoverage}]\label{proof:entropy-coverage}
We use \cref{prop:coverage_condition}.
 For any $T \subseteq [n]$,  define $x_T := I(Y_T|Y_{\bar{T}})$, where $Y_{\bar{T}} = [n] \setminus T$ and $I$ is the multivariate mutual information. We prove inductively that $H(Y_S |Z ) = \sum_{T: T\cap S \neq \emptyset} I(Y_T| Y_{\bar{T}}, Z)$ for an arbitrary set of random variables $Z$. This would  imply that for any  $S \subseteq [n]$, $f(S) =H(Y_S) =  \sum_{T: T\cap S \neq \emptyset} x_T$, which would finish the proof.  When $n =1$, the statement trivially holds. Assuming the statement is true for $n = k-1$, we prove it for for $n=k$. First, let $S=\{i\}$. We have
 \[
\sum_{T: T\cap S \neq \emptyset} I(Y_T \mid Y_{\bar{T}}, Z)= \sum_{T: i \in T} I(Y_T \mid Y_{\bar{T}}, Z) = (\sum_{T: i \in T, T \neq \{i\}} I(Y_{T \setminus \{i\}} | Y_{\bar{T}}, Z)  -  I(Y_{T \setminus \{i\}} | Y_{\bar{T} \cup \{i\}}) , Z)+ H(Y_i | Y_{\bar{i}}, Z),
\]
where we used the  fact that for any set of random variables $X_1, \dots, X_k$, $k\geq 2$, and any set of random variables $Z$, 
\[
I(X_1, \dots, X_k|Z) :=  I(X_1, \dots, X_{k-1}|Z) - I(X_1, \dots, X_{k-1}| X_k, Z).
\]
Using induction hypothesis
\begin{align*}
\sum_{T: i \in T, T \neq \{i\}} I(Y_{T \setminus \{i\}} | Y_{\bar{T}}, Z)  -  I(Y_{T \setminus \{i\}} | Y_{\bar{T} \cup \{i\}}, Z) &= \sum_{T\subseteq [n] \setminus \{i\}: T \cap ([n]\setminus \{i\})\neq \emptyset } I(Y_{T} | Y_{[n]\setminus(T \cup \{i\})}, Z)  -  I(Y_{T} | Y_{[n]\setminus T}| Z)  \\& =  H(Y_{[n] \setminus \{i \}}| Z) -  H(Y_{[n] \setminus\{ i\} }| Y_{i}, Z)\\ & = H(Y_{[n] \setminus \{i\}}| Z) - H(Y_{[n]}|Z) + H(Y_i|Z)  \\ & = -H(Y_i | Y_{\bar{i}}, Z) +H(Y_i |Z) 
\end{align*}
Therefore, 
\begin{align*}
\sum_{T: T\cap S \neq \emptyset} I(Y_T \mid Y_{\bar{T}}, Z) = H(Y_i|Z). 
\end{align*}

Now, assume that this equation holds for any  $S$ such that $|S| <l$. We want to show that it holds for $|S| = l$.  Choose $i \in S$.  We have
\[
\sum_{T: T\cap S \neq \emptyset} I(Y_T \mid Y_{\bar{T}}, Z) =\sum_{T: T\cap ( S \setminus \{i\}) \neq \emptyset} I(Y_T | Y_{\bar{T}}, Z) + \sum_{T \cap S = \{i\}}  I(Y_T | Y_{\bar{T}}, Z).
\]
Note that $\sum_{T: T\cap ( S \setminus \{i\}) \neq \emptyset} I(Y_T | Y_{\bar{T}}, Z) = H(Y_{S \setminus \{i\}}|Z) $ by  the second induction hypothesis. Furthermore,
\[
 \sum_{T \cap S = \{i\}}  I(Y_T | Y_{\bar{T}}, Z) =  (\sum_{T \cap S = \{i\}, T \neq \{i\}}   I(Y_{T \setminus \{i\}} | Y_{\bar{T}}, Z)  -  I(Y_{T \setminus \{i\}} | Y_{\bar{T} \cup \{i\}}, Z) )+ H(Y_i | Y_{\bar{i}}).
\]
Similar to what we did before, using the first induction hypothesis we get
\begin{align*}
\sum_{T \cap S = \{i\}, T \neq \{i\}}   I(Y_{T \setminus \{i\}} | Y_{\bar{T}}, Z)  -  I(Y_{T \setminus \{i\}} | Y_{\bar{T} \cup \{i\}}, Z) &= 
H(Y_{[n] \setminus S } \mid Y_{S \setminus \{i\}},  Z) -  H(Y_{[n] \setminus S }| Y_{S}, Z) \\
&=  H(Y_{[n] \setminus S } \mid Y_{S \setminus \{i\}}, Z) \\&-  H(Y_{([n] \setminus S) \cup \{i\}} \mid Y_{S \setminus \{i\}}, Z) + H( Y_{ i} \mid Y_{S \setminus \{i\}}, Z)
\\
&= -H(Y_i | Y_{\bar{i}}, Z)+H( Y_{ i} \mid Y_{S \setminus \{i\}}, Z).
\end{align*}
Therefore, 
\[
\sum_{T \cap S = \{i\}}  I(Y_T | Y_{\bar{T}}, Z)  = H( Y_{ i} \mid Y_{S \setminus \{i\}}, Z).
\]
So 
\[
\sum_{T: T\cap S \neq \emptyset} I(Y_T \mid Y_{\bar{T}}, Z) =H( Y_{ i} \mid Y_{S \setminus \{i\}}, Z)+  H(Y_{S \setminus \{i\}}, Z)  = H(Y_S|Z), 
\]
as desired. 
\end{proof}